\documentclass[11pt]{article}
\usepackage[utf8]{inputenc}
\usepackage[T1]{fontenc}
\usepackage{graphicx}
\usepackage{longtable}
\usepackage{rotating}
\usepackage[normalem]{ulem}
\usepackage{amsmath}
\usepackage{amssymb}
\usepackage{capt-of}
\usepackage{hyperref}

\usepackage{times} % times new roman font
\usepackage{xcolor} % for farger i xfig generert latex-kode
\usepackage{amsthm} % for å definere teorem environmthere is an algorithm for translating any quantifier-free first-order $L_{n,m}$ exthere is an algorithm for translating any quantifier-free first-order $L_{n,m}$ exthere is an algorithm for translating any quantifier-free first-order $L_{n,m}$ exthere is an algorithm for translating any quantifier-free first-order $L_{n,m}$ expression into a flat expressionpression into a flat expressionpression into a flat expressionpression into a flat expressionent o.l.

\theoremstyle{plain}
\newtheorem{theorem}{Theorem}[section]
\newtheorem{lemma}{Lemma}[section]
\newtheorem{corrollary}{Corrollary}[section]

\theoremstyle{definition}
\newtheorem{definition}{Definition}[section]
\newtheorem{example}{Example}[section]

\date{\today}
\title{A generalization of a representation of the integers modulo $p$, for the purpose of occasionally establishing the unsolvability of diophantine inequalities}
\author{André Rognes \\
  \small Oslo Metropolitan University \\
  \small \texttt{andre.rognes@oslomet.no}
}

\hypersetup{
 pdfauthor={André Rognes},
 pdftitle={A generalization of a representation of the integers modulo $n$},
 pdfkeywords={},
 pdfsubject={},
 pdfcreator={Emacs 30.2 (Org mode 9.7.11)}, 
 pdflang={English}}

\begin{document}

\maketitle

\begin{abstract} It is well known that if a diophantine equation turns out not to have a solution over the integers modulo $p$, for some $p$, then it does not have a solution over the integers per se. This is because the integers modulo $p$ are a homomorphic image of the integers. However, the integers modulo $p$ are of little use when faced with diophantine inequalities, as the homomorphic image of the less-than-relation is trivial. The purpose of the present paper is to introduce a way of gereralising a particular representation of the integers modulo $p$. The generalizations, novel to this paper, are in the form of decidable Lindenbaum-algebras, and allow for deciding whether given positive first-order formulas in the language of first-order arithmetic are solvable. Crucially if a system of diophantine inequalities turns out not to be solvable in one of the Lindenbaum-algebras, then it is not solvable over the standard integers.
\end{abstract}

03C90 Nonclassical models (Boolean-valued, sheaf, etc.), \\
03B25 Decidability of theories and sets of sentences \\
11D75 Diophantine inequalities, \\

\section{Introduction} It is well known that if a diophantine equation turns out not to have a solution over the integers modulo $p$, for some $p$, then it does not have a solution over the integers per se. This is because the integers modulo $p$ are a homomorphic image of the integers. However, the integers modulo $p$ are of little use for inequalities as the homomorphic image of the less-than-relation is trivial, i.e. all numbers are greater (and less) than every other number modulo $p$. The purpose of the present paper is to introduce a way of gereralising a particular representation of the integers modulo $p$. The generalizations, novel to this paper, rely on model-theoretic techniques and allow for deciding whether given positive first-order formulas in the language of first-order arithmetic are solvable in the generalized representation. Crucially if a system of diophantine inequalities turns out not to be solvable over one of the generalized representations, then it is not solvable over the standard integers. We think of the generalised representations as approximations to addition and multiplication of integers and for each $p$ there is a family of generalizations that lie between the integers and the integers modulo $p$.
%% If a diophantine equation that does not have a solution over the integers modulo $p$ it will not have a solutions over any of the generalizations between the integers and those modulo $p$.

\subsection{Formal languages}
For each natural number $p$ greater than one, not neccearily prime, we define two languages one more expressive than the other. The standard interpretation of the less expressive language is decidable whilst the other language contains diophantine equations and therefore the standard interpretation is undecidable, by the Matiyasevich – Robinson – Davis – Putnam theorem, \cite{zbMATH00545277}. 
\begin{definition}
Let $L_p^{min}$ be the first-order language based on the following:
\begin{description}
%% \item variables: $v_0, \ldots, v_{n-1}$ (we also use $x,y,z,u,v,w$ for readability)
\item a numeral for each for each natural number
\item a unary function symbol: $-$, used for negation
\item unary relation symbols: $R_0, \ldots, R_{p-1}$, used for residue classes mod $p$
\item a binary relation symbol: $<$, used for less than
\item a binary relation symbol: $=$, used for equality
\end{description}
\end{definition}
So the {\em standard interpretation} of this language maps each numeral
to the corresponding number, negation, $<$ and $=$ are as in school
and $R_0, R_1, \ldots$ denotes the set of integers modulo $0$, the set of
integers modulo $1$, ... respectively.
\begin{lemma} The standard interpretation of $L_p^{min}$ is decidable.
\end{lemma}
\textbf{Proof:} Consider finite state machines that work over
$n$-tuples of signed numerals in base $p$, i.e. $(\{+,-\} \times p^*)^n$.
For each atomic formula it is sraight forward to construct an automaton
that reccognizes exactly the $n$-tuples that represent the variable
instansiations true of the formula.
Moreover it is well-known that once we have automata for each atomic formula
of a given first-order formula, then for each sub-formula we can construct
a finite state machine that recognises the corresponding interpretation.
\hfill Q.E.D.

\begin{definition}
We write $L_p$ for the language $L_p^{min}$ expanded with a
\begin{description}
\item a binary function symbol: $+$, used for addition
\item a binary function symbol: $\cdot$, used for multiplication
\end{description}
\end{definition}
We regard relation symbols such as $>$, $\leq$ and $\geq$ as  derived from
$<$ and $=$.
\begin{definition} A {\em positive formula} is a formula whose quantifiers and connectives are all amongst $\forall$, $\exists$, $\vee$ and $\wedge$. In particular they do not make use of negation or a combination of connectives that allowes us to express negation.
\end{definition}
In the standard interpretation of $L_p^{min}$ we may express
non-possitive formulas as positive ones. For example $x \neq y$ is equivalent to $x < y \vee y < x$.

In order to give the term diophantine inequality a presice meaning we add the following.
\begin{definition} We use the term {\em system of diophantine inequalities} or just {\em diophantine inequality} for positive formulas in the language $L_p$.
\end{definition}

Throughout we will interpret these languages in sub-algebras of the Lindenbaum algebra of the theory of formulas valid in the standard interpretation of $L_p$.
\subsection{Lindenbaum and Tarski}
\begin{definition} Let $Z$ be the set of $L_p$ formulas valid in the standard interpretation, i.e. true in every instansiation of the variables with integers. Denote by $L_p/Z$ the set of equvalence classes under the equivalence realation where $[\phi]$ and $[\psi]$ are the same if the bi-implication $\phi \leftrightarrow \psi$ is in $Z$. We refer to  $L_p/Z$ as the {\em Lindenbaum algebra of Z}.
\end{definition} 

The Lindenbaum algebra and the following few facts about it are well known. The mapping $\phi \mapsto [\phi]$ is a homomorphims with respect to the first-order connectives, for example  $[\exists x \phi] = \exists x [\phi]$. Moreover there is an isomorphism between the Lindenbaum algebra of $Z$ and the power set of instansiations of the variables in $L_Z$ with integers. The definition of the latter power set is due to A. Tarski, it underlies most of model-theory and we refer to mappings from our language to it as {\em tarskian interpretations}. Here, for example $[\phi \rightarrow \psi]$ corresponds to the interpretation of $\phi \rightarrow \psi$ which is valid iff $[\phi] < [\psi]$. We say that $[\phi]$ {\em is below} $[\psi]$ is this case. This in turn corresponds to the tarskian interpretation of $\phi$ being a subset of the interpretation of $\psi$. In the case of first-order quantifiers, this isomorphim maps $[\exists x \phi]$ to the least (hyper-) cylinder containing the interpretation of $\phi$.

\subsection{Flat formulas}
We are interested in the graphs of addition and multiplication. Also we are interested in relations that contain these graphs. Rather than introducing, a large number of relation symbols into our language we use the following notion.
\begin{definition}
$L_p$ formulas (including $L_p^{min}$ formulas) are said to be \emph{flat} if they are constructed from atomic formulas of one of the following 10 forms, using connectives (i. e. $\vee$,$\wedge$,$\neg$) and first-order quantifiers.
  \begin{description}
  \item $x = y$, $x < y$,
  \item $a = x$, $a < x$, for each numeral $a$,
   \item $-(x) = y$, $-(x) < y$,
  \item $x + y = z$, $x + y < z$,
  \item $x \cdot y = z$, $x \cdot y < z$.
\end{description}
\end{definition}

\begin{lemma} There is an algorithm for translating any quantifier-free first-order $L_p$ formula into a flat formula.
\end{lemma}
\textbf{Proof:} This can be proven by induction on the buildup of formulas.
\hfill Q.E.D.

Since there is an algorithm for translating any first-order $L_p$ formula into a flat formula, these extend to interpretations of systems of diophantine (in-)equalities with coefficients in $I$.

\section{Interpretations}
We will now be defining ways of interpreting flat formulas that are, so to speak, approximations to the standard, or precise, interpretation. Due to the incompleteness theorem the precise interpretation of flat $L_p$ formulas in general is undecidable. The undecidability of positive $L_p$ formulas follows by the Matiyasevich – Robinson – Davis – Putnam theorem, \cite{zbMATH00545277}. Interpretations over the integers modulo $p$ are the least precise interpretations in this contexts. Therebetween are decidable interpretations that can be used to gain insight into solvabilitiy questions of diophantine inequalities.

\begin{definition} The {\em precise interpretation} is the mapping that maps each formula $\phi \in L_p$ to $[ \phi ]$ in $L_p/Z$
\end{definition}

In order to define interpretations of systems of diophantine inequalities that lie between the precise interpretation and interpretation over the integers modulo $p$ we introduce interpretations of addition and multiplication that lie between the standard interpretations and the interpretations over the integers modulo $p$. From now on we shall be using the symbol $\sigma$ for approximations of summation and $\pi$ for approximations of the product, more spesicically the graphs thereof.
\begin{definition}
Let $\sigma$ and $\pi$ be $L_p^{min}$ formulas with three free variables. Then $\sigma(x,y,z)$ and $\pi(x,y,z)$ are said to be {\em approxiamtions to the sum and product} if their interpretations lie above the sum and product in $L_p/Z$, i.e. if $[x + y = z] \leq [\sigma(x,y,z)]$ and $[x \cdot y = z] \leq [\pi(x,y,z)]$ 
\end{definition}
\begin{figure}[ht]
  \begin{center}
%%    \input{xfig/addition}
%
%  encoding: UTF-8 
%  encoding: UTF-8 
%
\setlength{\unitlength}{3947sp}%
\begin{picture}(4227,5047)(1189,-4198)
\thinlines
{\color[rgb]{0,0,0}\put(2101,-3601){\makebox(1.6667,11.6667){\small.}}
}%
{\color[rgb]{0,0,0}\put(2162,-3541){\makebox(1.6667,11.6667){\small.}}
}%
{\color[rgb]{0,0,0}\put(2222,-3481){\makebox(1.6667,11.6667){\small.}}
}%
{\color[rgb]{0,0,0}\put(4291,-3376){\makebox(1.6667,11.6667){\small.}}
}%
{\color[rgb]{0,0,0}\put(4389,-3481){\makebox(1.6667,11.6667){\small.}}
}%
{\color[rgb]{0,0,0}\put(4501,-3586){\makebox(1.6667,11.6667){\small.}}
}%
{\color[rgb]{0,0,0}\put(2221,-1232){\makebox(1.6667,11.6667){\small.}}
}%
{\color[rgb]{0,0,0}\put(2161,-1171){\makebox(1.6667,11.6667){\small.}}
}%
{\color[rgb]{0,0,0}\put(2101,-1111){\makebox(1.6667,11.6667){\small.}}
}%
{\color[rgb]{0,0,0}\put(4501,-1253){\makebox(1.6667,11.6667){\small.}}
}%
{\color[rgb]{0,0,0}\put(4704,-1059){\makebox(1.6667,11.6667){\small.}}
}%
{\color[rgb]{0,0,0}\put(4606,-1148){\makebox(1.6667,11.6667){\small.}}
}%
{\color[rgb]{0,0,0}\put(4209,837){\makebox(1.6667,11.6667){\small.}}
}%
{\color[rgb]{0,0,0}\put(2251,-2386){\line(-1, 0){1050}}
}%
{\color[rgb]{0,0,0}\put(4426,-2386){\vector( 1, 0){825}}
}%
{\color[rgb]{0,0,0}\put(3301,-1186){\vector( 0, 1){450}}
}%
{\color[rgb]{0,0,0}\put(3301,-3436){\line( 0,-1){750}}
}%
\put(3301,-2461){\makebox(0,0)[lb]{\smash{\fontsize{12}{14.4}\usefont{T1}{ptm}{m}{n}{\color[rgb]{0,0,0}0}%
}}}
\put(3601,-2461){\makebox(0,0)[lb]{\smash{\fontsize{12}{14.4}\usefont{T1}{ptm}{m}{n}{\color[rgb]{0,0,0}1}%
}}}
\put(3901,-2461){\makebox(0,0)[lb]{\smash{\fontsize{12}{14.4}\usefont{T1}{ptm}{m}{n}{\color[rgb]{0,0,0}2}%
}}}
\put(3301,-2161){\makebox(0,0)[lb]{\smash{\fontsize{12}{14.4}\usefont{T1}{ptm}{m}{n}{\color[rgb]{0,0,0}1}%
}}}
\put(3601,-2161){\makebox(0,0)[lb]{\smash{\fontsize{12}{14.4}\usefont{T1}{ptm}{m}{n}{\color[rgb]{0,0,0}2}%
}}}
\put(3901,-2161){\makebox(0,0)[lb]{\smash{\fontsize{12}{14.4}\usefont{T1}{ptm}{m}{n}{\color[rgb]{0,0,0}3}%
}}}
\put(3001,-1861){\makebox(0,0)[lb]{\smash{\fontsize{12}{14.4}\usefont{T1}{ptm}{m}{n}{\color[rgb]{0,0,0}1}%
}}}
\put(2701,-1861){\makebox(0,0)[lb]{\smash{\fontsize{12}{14.4}\usefont{T1}{ptm}{m}{n}{\color[rgb]{0,0,0}0}%
}}}
\put(3301,-1861){\makebox(0,0)[lb]{\smash{\fontsize{12}{14.4}\usefont{T1}{ptm}{m}{n}{\color[rgb]{0,0,0}2}%
}}}
\put(3601,-1861){\makebox(0,0)[lb]{\smash{\fontsize{12}{14.4}\usefont{T1}{ptm}{m}{n}{\color[rgb]{0,0,0}3}%
}}}
\put(3901,-1861){\makebox(0,0)[lb]{\smash{\fontsize{12}{14.4}\usefont{T1}{ptm}{m}{n}{\color[rgb]{0,0,0}4}%
}}}
\put(3601,-2761){\makebox(0,0)[lb]{\smash{\fontsize{12}{14.4}\usefont{T1}{ptm}{m}{n}{\color[rgb]{0,0,0}0}%
}}}
\put(3901,-2761){\makebox(0,0)[lb]{\smash{\fontsize{12}{14.4}\usefont{T1}{ptm}{m}{n}{\color[rgb]{0,0,0}1}%
}}}
\put(3901,-3061){\makebox(0,0)[lb]{\smash{\fontsize{12}{14.4}\usefont{T1}{ptm}{m}{n}{\color[rgb]{0,0,0}0}%
}}}
\put(2626,-2161){\makebox(0,0)[lb]{\smash{\fontsize{12}{14.4}\usefont{T1}{ptm}{m}{n}{\color[rgb]{0,0,0}-1}%
}}}
\put(2626,-2461){\makebox(0,0)[lb]{\smash{\fontsize{12}{14.4}\usefont{T1}{ptm}{m}{n}{\color[rgb]{0,0,0}-2}%
}}}
\put(2626,-2761){\makebox(0,0)[lb]{\smash{\fontsize{12}{14.4}\usefont{T1}{ptm}{m}{n}{\color[rgb]{0,0,0}-3}%
}}}
\put(2626,-3061){\makebox(0,0)[lb]{\smash{\fontsize{12}{14.4}\usefont{T1}{ptm}{m}{n}{\color[rgb]{0,0,0}-4}%
}}}
\put(2926,-2461){\makebox(0,0)[lb]{\smash{\fontsize{12}{14.4}\usefont{T1}{ptm}{m}{n}{\color[rgb]{0,0,0}-1}%
}}}
\put(2926,-2761){\makebox(0,0)[lb]{\smash{\fontsize{12}{14.4}\usefont{T1}{ptm}{m}{n}{\color[rgb]{0,0,0}-2}%
}}}
\put(2926,-3061){\makebox(0,0)[lb]{\smash{\fontsize{12}{14.4}\usefont{T1}{ptm}{m}{n}{\color[rgb]{0,0,0}-3}%
}}}
\put(3226,-2761){\makebox(0,0)[lb]{\smash{\fontsize{12}{14.4}\usefont{T1}{ptm}{m}{n}{\color[rgb]{0,0,0}-1}%
}}}
\put(3226,-3061){\makebox(0,0)[lb]{\smash{\fontsize{12}{14.4}\usefont{T1}{ptm}{m}{n}{\color[rgb]{0,0,0}-2}%
}}}
\put(3526,-3061){\makebox(0,0)[lb]{\smash{\fontsize{12}{14.4}\usefont{T1}{ptm}{m}{n}{\color[rgb]{0,0,0}-1}%
}}}
\put(3001,-2161){\makebox(0,0)[lb]{\smash{\fontsize{12}{14.4}\usefont{T1}{ptm}{m}{n}{\color[rgb]{0,0,0}0}%
}}}
\put(5401,-2461){\makebox(0,0)[lb]{\smash{\fontsize{12}{14.4}\usefont{T1}{ptm}{m}{n}{\color[rgb]{0,0,0}x}%
}}}
\put(3301,-586){\makebox(0,0)[lb]{\smash{\fontsize{12}{14.4}\usefont{T1}{ptm}{m}{n}{\color[rgb]{0,0,0}y}%
}}}
\put(4201,-1861){\makebox(0,0)[lb]{\smash{\fontsize{12}{14.4}\usefont{T1}{ptm}{m}{n}{\color[rgb]{0,0,0}5}%
}}}
\put(3001,-1561){\makebox(0,0)[lb]{\smash{\fontsize{12}{14.4}\usefont{T1}{ptm}{m}{n}{\color[rgb]{0,0,0}2}%
}}}
\put(3301,-1561){\makebox(0,0)[lb]{\smash{\fontsize{12}{14.4}\usefont{T1}{ptm}{m}{n}{\color[rgb]{0,0,0}3}%
}}}
\put(3601,-1561){\makebox(0,0)[lb]{\smash{\fontsize{12}{14.4}\usefont{T1}{ptm}{m}{n}{\color[rgb]{0,0,0}4}%
}}}
\put(3901,-1561){\makebox(0,0)[lb]{\smash{\fontsize{12}{14.4}\usefont{T1}{ptm}{m}{n}{\color[rgb]{0,0,0}5}%
}}}
\put(4201,-1561){\makebox(0,0)[lb]{\smash{\fontsize{12}{14.4}\usefont{T1}{ptm}{m}{n}{\color[rgb]{0,0,0}6}%
}}}
\put(4201,-2161){\makebox(0,0)[lb]{\smash{\fontsize{12}{14.4}\usefont{T1}{ptm}{m}{n}{\color[rgb]{0,0,0}4}%
}}}
\put(4201,-2761){\makebox(0,0)[lb]{\smash{\fontsize{12}{14.4}\usefont{T1}{ptm}{m}{n}{\color[rgb]{0,0,0}2}%
}}}
\put(2701,-1561){\makebox(0,0)[lb]{\smash{\fontsize{12}{14.4}\usefont{T1}{ptm}{m}{n}{\color[rgb]{0,0,0}1}%
}}}
\put(3226,-3361){\makebox(0,0)[lb]{\smash{\fontsize{12}{14.4}\usefont{T1}{ptm}{m}{n}{\color[rgb]{0,0,0}-3}%
}}}
\put(3526,-3361){\makebox(0,0)[lb]{\smash{\fontsize{12}{14.4}\usefont{T1}{ptm}{m}{n}{\color[rgb]{0,0,0}-2}%
}}}
\put(3826,-3361){\makebox(0,0)[lb]{\smash{\fontsize{12}{14.4}\usefont{T1}{ptm}{m}{n}{\color[rgb]{0,0,0}-1}%
}}}
\put(4201,-3361){\makebox(0,0)[lb]{\smash{\fontsize{12}{14.4}\usefont{T1}{ptm}{m}{n}{\color[rgb]{0,0,0}0}%
}}}
\put(2401,-1561){\makebox(0,0)[lb]{\smash{\fontsize{12}{14.4}\usefont{T1}{ptm}{m}{n}{\color[rgb]{0,0,0}0}%
}}}
\put(2326,-1861){\makebox(0,0)[lb]{\smash{\fontsize{12}{14.4}\usefont{T1}{ptm}{m}{n}{\color[rgb]{0,0,0}-1}%
}}}
\put(2326,-2161){\makebox(0,0)[lb]{\smash{\fontsize{12}{14.4}\usefont{T1}{ptm}{m}{n}{\color[rgb]{0,0,0}-2}%
}}}
\put(2326,-2461){\makebox(0,0)[lb]{\smash{\fontsize{12}{14.4}\usefont{T1}{ptm}{m}{n}{\color[rgb]{0,0,0}-3}%
}}}
\put(2326,-2761){\makebox(0,0)[lb]{\smash{\fontsize{12}{14.4}\usefont{T1}{ptm}{m}{n}{\color[rgb]{0,0,0}-4}%
}}}
\put(2326,-3061){\makebox(0,0)[lb]{\smash{\fontsize{12}{14.4}\usefont{T1}{ptm}{m}{n}{\color[rgb]{0,0,0}-5}%
}}}
\put(2326,-3361){\makebox(0,0)[lb]{\smash{\fontsize{12}{14.4}\usefont{T1}{ptm}{m}{n}{\color[rgb]{0,0,0}-6}%
}}}
\put(2626,-3361){\makebox(0,0)[lb]{\smash{\fontsize{12}{14.4}\usefont{T1}{ptm}{m}{n}{\color[rgb]{0,0,0}-5}%
}}}
\put(2926,-3361){\makebox(0,0)[lb]{\smash{\fontsize{12}{14.4}\usefont{T1}{ptm}{m}{n}{\color[rgb]{0,0,0}-4}%
}}}
\put(4201,-2461){\makebox(0,0)[lb]{\smash{\fontsize{12}{14.4}\usefont{T1}{ptm}{m}{n}{\color[rgb]{0,0,0}3}%
}}}
\put(4201,-3061){\makebox(0,0)[lb]{\smash{\fontsize{12}{14.4}\usefont{T1}{ptm}{m}{n}{\color[rgb]{0,0,0}1}%
}}}
\end{picture}%

%%%%%%%%%%%%%%%%%%%%%%%%%%%%
  \end{center}
  \caption{The addition table with regions compising values of $z$ that are all equal, less or grater than $0$, $x$ or $y$}
  \label{fig:addition}
\end{figure}

\begin{figure}[ht]
  \begin{center}  
%%    \input{xfig/multiplication}

%
%  encoding: UTF-8 
%
\setlength{\unitlength}{3947sp}%
\begingroup\makeatletter\ifx\SetFigFont\undefined%
\gdef\SetFigFont#1#2#3#4#5{%
  \reset@font\fontsize{#1}{#2pt}%
  \fontfamily{#3}\fontseries{#4}\fontshape{#5}%
  \selectfont}%
\fi\endgroup%
\begin{picture}(3552,3507)(1639,-3958)
\thinlines
{\color[rgb]{0,0,0}\put(2026,-2386){\line(-1, 0){375}}
}%
{\color[rgb]{0,0,0}\put(3346,-3571){\line( 0,-1){375}}
}%
{\color[rgb]{0,0,0}\put(3338,-1111){\vector( 0, 1){450}}
}%
{\color[rgb]{0,0,0}\put(4426,-1336){\makebox(1.6667,11.6667){\small.}}
}%
{\color[rgb]{0,0,0}\put(4501,-1261){\makebox(1.6667,11.6667){\small.}}
}%
{\color[rgb]{0,0,0}\put(4576,-1186){\makebox(1.6667,11.6667){\small.}}
}%
{\color[rgb]{0,0,0}\put(4426,-3436){\makebox(1.6667,11.6667){\small.}}
}%
{\color[rgb]{0,0,0}\put(4501,-3511){\makebox(1.6667,11.6667){\small.}}
}%
{\color[rgb]{0,0,0}\put(4576,-3586){\makebox(1.6667,11.6667){\small.}}
}%
{\color[rgb]{0,0,0}\put(2176,-3436){\makebox(1.6667,11.6667){\small.}}
}%
{\color[rgb]{0,0,0}\put(2101,-3511){\makebox(1.6667,11.6667){\small.}}
}%
{\color[rgb]{0,0,0}\put(2026,-3586){\makebox(1.6667,11.6667){\small.}}
}%
{\color[rgb]{0,0,0}\put(2176,-1336){\makebox(1.6667,11.6667){\small.}}
}%
{\color[rgb]{0,0,0}\put(2101,-1261){\makebox(1.6667,11.6667){\small.}}
}%
{\color[rgb]{0,0,0}\put(2026,-1186){\makebox(1.6667,11.6667){\small.}}
}%
{\color[rgb]{0,0,0}\put(4651,-2386){\vector( 1, 0){450}}
}%
\put(3301,-2461){\makebox(0,0)[lb]{\smash{{\SetFigFont{12}{14.4}{\rmdefault}{\mddefault}{\updefault}{\color[rgb]{0,0,0}0}%
}}}}
\put(3901,-1861){\makebox(0,0)[lb]{\smash{{\SetFigFont{12}{14.4}{\rmdefault}{\mddefault}{\updefault}{\color[rgb]{0,0,0}4}%
}}}}
\put(3301,-586){\makebox(0,0)[lb]{\smash{{\SetFigFont{12}{14.4}{\rmdefault}{\mddefault}{\updefault}{\color[rgb]{0,0,0}y}%
}}}}
\put(2626,-2461){\makebox(0,0)[lb]{\smash{{\SetFigFont{12}{14.4}{\rmdefault}{\mddefault}{\updefault}{\color[rgb]{0,0,0}0}%
}}}}
\put(2926,-2461){\makebox(0,0)[lb]{\smash{{\SetFigFont{12}{14.4}{\rmdefault}{\mddefault}{\updefault}{\color[rgb]{0,0,0}0}%
}}}}
\put(3601,-2461){\makebox(0,0)[lb]{\smash{{\SetFigFont{12}{14.4}{\rmdefault}{\mddefault}{\updefault}{\color[rgb]{0,0,0}0}%
}}}}
\put(3901,-2461){\makebox(0,0)[lb]{\smash{{\SetFigFont{12}{14.4}{\rmdefault}{\mddefault}{\updefault}{\color[rgb]{0,0,0}0}%
}}}}
\put(3301,-2161){\makebox(0,0)[lb]{\smash{{\SetFigFont{12}{14.4}{\rmdefault}{\mddefault}{\updefault}{\color[rgb]{0,0,0}0}%
}}}}
\put(3301,-1861){\makebox(0,0)[lb]{\smash{{\SetFigFont{12}{14.4}{\rmdefault}{\mddefault}{\updefault}{\color[rgb]{0,0,0}0}%
}}}}
\put(3601,-2161){\makebox(0,0)[lb]{\smash{{\SetFigFont{12}{14.4}{\rmdefault}{\mddefault}{\updefault}{\color[rgb]{0,0,0}1}%
}}}}
\put(3601,-1861){\makebox(0,0)[lb]{\smash{{\SetFigFont{12}{14.4}{\rmdefault}{\mddefault}{\updefault}{\color[rgb]{0,0,0}2}%
}}}}
\put(3901,-2161){\makebox(0,0)[lb]{\smash{{\SetFigFont{12}{14.4}{\rmdefault}{\mddefault}{\updefault}{\color[rgb]{0,0,0}2}%
}}}}
\put(2926,-2761){\makebox(0,0)[lb]{\smash{{\SetFigFont{12}{14.4}{\rmdefault}{\mddefault}{\updefault}{\color[rgb]{0,0,0}1}%
}}}}
\put(2926,-3061){\makebox(0,0)[lb]{\smash{{\SetFigFont{12}{14.4}{\rmdefault}{\mddefault}{\updefault}{\color[rgb]{0,0,0}2}%
}}}}
\put(2626,-2761){\makebox(0,0)[lb]{\smash{{\SetFigFont{12}{14.4}{\rmdefault}{\mddefault}{\updefault}{\color[rgb]{0,0,0}2}%
}}}}
\put(2626,-3061){\makebox(0,0)[lb]{\smash{{\SetFigFont{12}{14.4}{\rmdefault}{\mddefault}{\updefault}{\color[rgb]{0,0,0}4}%
}}}}
\put(2551,-2161){\makebox(0,0)[lb]{\smash{{\SetFigFont{12}{14.4}{\rmdefault}{\mddefault}{\updefault}{\color[rgb]{0,0,0}-2}%
}}}}
\put(2551,-1861){\makebox(0,0)[lb]{\smash{{\SetFigFont{12}{14.4}{\rmdefault}{\mddefault}{\updefault}{\color[rgb]{0,0,0}-4}%
}}}}
\put(2851,-2161){\makebox(0,0)[lb]{\smash{{\SetFigFont{12}{14.4}{\rmdefault}{\mddefault}{\updefault}{\color[rgb]{0,0,0}-1}%
}}}}
\put(2851,-1861){\makebox(0,0)[lb]{\smash{{\SetFigFont{12}{14.4}{\rmdefault}{\mddefault}{\updefault}{\color[rgb]{0,0,0}-2}%
}}}}
\put(3301,-3061){\makebox(0,0)[lb]{\smash{{\SetFigFont{12}{14.4}{\rmdefault}{\mddefault}{\updefault}{\color[rgb]{0,0,0}0}%
}}}}
\put(3301,-2761){\makebox(0,0)[lb]{\smash{{\SetFigFont{12}{14.4}{\rmdefault}{\mddefault}{\updefault}{\color[rgb]{0,0,0}0}%
}}}}
\put(3526,-3061){\makebox(0,0)[lb]{\smash{{\SetFigFont{12}{14.4}{\rmdefault}{\mddefault}{\updefault}{\color[rgb]{0,0,0}-2}%
}}}}
\put(3526,-2761){\makebox(0,0)[lb]{\smash{{\SetFigFont{12}{14.4}{\rmdefault}{\mddefault}{\updefault}{\color[rgb]{0,0,0}-1}%
}}}}
\put(3826,-2761){\makebox(0,0)[lb]{\smash{{\SetFigFont{12}{14.4}{\rmdefault}{\mddefault}{\updefault}{\color[rgb]{0,0,0}-2}%
}}}}
\put(3826,-3061){\makebox(0,0)[lb]{\smash{{\SetFigFont{12}{14.4}{\rmdefault}{\mddefault}{\updefault}{\color[rgb]{0,0,0}-4}%
}}}}
\put(3601,-1561){\makebox(0,0)[lb]{\smash{{\SetFigFont{12}{14.4}{\rmdefault}{\mddefault}{\updefault}{\color[rgb]{0,0,0}3}%
}}}}
\put(3901,-1561){\makebox(0,0)[lb]{\smash{{\SetFigFont{12}{14.4}{\rmdefault}{\mddefault}{\updefault}{\color[rgb]{0,0,0}6}%
}}}}
\put(4201,-1561){\makebox(0,0)[lb]{\smash{{\SetFigFont{12}{14.4}{\rmdefault}{\mddefault}{\updefault}{\color[rgb]{0,0,0}9}%
}}}}
\put(4201,-1861){\makebox(0,0)[lb]{\smash{{\SetFigFont{12}{14.4}{\rmdefault}{\mddefault}{\updefault}{\color[rgb]{0,0,0}6}%
}}}}
\put(4201,-2161){\makebox(0,0)[lb]{\smash{{\SetFigFont{12}{14.4}{\rmdefault}{\mddefault}{\updefault}{\color[rgb]{0,0,0}3}%
}}}}
\put(4201,-2461){\makebox(0,0)[lb]{\smash{{\SetFigFont{12}{14.4}{\rmdefault}{\mddefault}{\updefault}{\color[rgb]{0,0,0}0}%
}}}}
\put(2251,-1561){\makebox(0,0)[lb]{\smash{{\SetFigFont{12}{14.4}{\rmdefault}{\mddefault}{\updefault}{\color[rgb]{0,0,0}-9}%
}}}}
\put(2551,-1561){\makebox(0,0)[lb]{\smash{{\SetFigFont{12}{14.4}{\rmdefault}{\mddefault}{\updefault}{\color[rgb]{0,0,0}-6}%
}}}}
\put(2851,-1561){\makebox(0,0)[lb]{\smash{{\SetFigFont{12}{14.4}{\rmdefault}{\mddefault}{\updefault}{\color[rgb]{0,0,0}-3}%
}}}}
\put(3301,-1561){\makebox(0,0)[lb]{\smash{{\SetFigFont{12}{14.4}{\rmdefault}{\mddefault}{\updefault}{\color[rgb]{0,0,0}0}%
}}}}
\put(2251,-1861){\makebox(0,0)[lb]{\smash{{\SetFigFont{12}{14.4}{\rmdefault}{\mddefault}{\updefault}{\color[rgb]{0,0,0}-6}%
}}}}
\put(2251,-2161){\makebox(0,0)[lb]{\smash{{\SetFigFont{12}{14.4}{\rmdefault}{\mddefault}{\updefault}{\color[rgb]{0,0,0}-3}%
}}}}
\put(2251,-2461){\makebox(0,0)[lb]{\smash{{\SetFigFont{12}{14.4}{\rmdefault}{\mddefault}{\updefault}{\color[rgb]{0,0,0}0}%
}}}}
\put(2251,-2761){\makebox(0,0)[lb]{\smash{{\SetFigFont{12}{14.4}{\rmdefault}{\mddefault}{\updefault}{\color[rgb]{0,0,0}3}%
}}}}
\put(2251,-3061){\makebox(0,0)[lb]{\smash{{\SetFigFont{12}{14.4}{\rmdefault}{\mddefault}{\updefault}{\color[rgb]{0,0,0}6}%
}}}}
\put(2251,-3361){\makebox(0,0)[lb]{\smash{{\SetFigFont{12}{14.4}{\rmdefault}{\mddefault}{\updefault}{\color[rgb]{0,0,0}9}%
}}}}
\put(2551,-3361){\makebox(0,0)[lb]{\smash{{\SetFigFont{12}{14.4}{\rmdefault}{\mddefault}{\updefault}{\color[rgb]{0,0,0}6}%
}}}}
\put(3301,-3361){\makebox(0,0)[lb]{\smash{{\SetFigFont{12}{14.4}{\rmdefault}{\mddefault}{\updefault}{\color[rgb]{0,0,0}0}%
}}}}
\put(2926,-3361){\makebox(0,0)[lb]{\smash{{\SetFigFont{12}{14.4}{\rmdefault}{\mddefault}{\updefault}{\color[rgb]{0,0,0}3}%
}}}}
\put(3526,-3361){\makebox(0,0)[lb]{\smash{{\SetFigFont{12}{14.4}{\rmdefault}{\mddefault}{\updefault}{\color[rgb]{0,0,0}-3}%
}}}}
\put(4126,-2761){\makebox(0,0)[lb]{\smash{{\SetFigFont{12}{14.4}{\rmdefault}{\mddefault}{\updefault}{\color[rgb]{0,0,0}-3}%
}}}}
\put(4126,-3061){\makebox(0,0)[lb]{\smash{{\SetFigFont{12}{14.4}{\rmdefault}{\mddefault}{\updefault}{\color[rgb]{0,0,0}-6}%
}}}}
\put(4126,-3361){\makebox(0,0)[lb]{\smash{{\SetFigFont{12}{14.4}{\rmdefault}{\mddefault}{\updefault}{\color[rgb]{0,0,0}-9}%
}}}}
\put(3826,-3361){\makebox(0,0)[lb]{\smash{{\SetFigFont{12}{14.4}{\rmdefault}{\mddefault}{\updefault}{\color[rgb]{0,0,0}-6}%
}}}}
\put(5176,-2461){\makebox(0,0)[lb]{\smash{{\SetFigFont{12}{14.4}{\rmdefault}{\mddefault}{\updefault}{\color[rgb]{0,0,0}x}%
}}}}
\end{picture}%

%%%%%%%%%%%%%%%%%%%%%%%%%%%%%%%%%%%%%%%%%%%%%%%%%%%%%%%%%%%%%%%%
  \end{center}
  \caption{The multiplication table with regions comprising values of $z$ that are all equal, less or grater than $0$, $x$ or $y$}
  \label{fig:multiplication}
\end{figure}

The following we may refer to as the trivial approximations. They give rise to a representation of the integers modulo $p$.
\begin{example} If $\sigma$ and $\pi$ are tautologies with three free varables, then they are approximations to the sum and product.
\end{example}
The next example allows us to interpret diophantine inequalities in a non-trivial way.
\begin{example} \label{minimal}
  Let $\sigma(x,y,z)$ and $\pi(x,y,z)$ be the following two formulas respecively.
  \begin{description}
  \item $((0 < y \wedge x \leq z) \vee (y = 0 \wedge x = z) \vee (y < 0 \wedge z \leq x))$. To convince ourselves that this is an approximation to the sum, the formula is designed so as to contain the graph of the addition function, see figure \ref{fig:addition}.
  \item $((0 < y \wedge x \leq z) \vee (y = 0 \wedge z = 0) \vee (y < 0 \wedge -(x) \leq z))$
    To convince ourselves that this is an approximation to the product, the formula $\pi$ is designed so as to contain the graph of the multiplication function, see figure \ref{fig:multiplication}
  \end{description}
\end{example}

Since addition and multiplication are commutative, in the previous example we can let $x$ and $y$ change roles. Moreover we can take the intersection of of
the previous example and the version with the role change as follows.
\begin{theorem} \label{intersection} The following two $L_p^{min}$-formulas are approximations to the sum and product, respectively.
  \begin{description}
\item $((0 < y \wedge x \leq z) \vee (y = 0 \wedge x = z) \vee (y < 0 \wedge z \leq x))$ $\wedge$ $((0 < x \wedge y \leq z) \vee (x = 0 \wedge y = z) \vee (x < 0 \wedge z \leq y))$
\item  $((0 < y \wedge x \leq z) \vee (y = 0 \wedge 0 = z) \vee (y < 0 \wedge -(x) \leq z))$
  $\wedge$
  $(0 < x \wedge y \leq z) \vee (x = 0 \wedge 0 = z) \vee (x < 0 \wedge -(y) \leq z)$
\end{description}
  \end{theorem}
\textbf{Proof:} Both of the approximations are conjunctions of two instansiations of the previous examples. Also the intersection of two approximations is an approximation.
\hfill{Q.E.D}

The following defines coarse interpretations, note that flat atomic formulas are interpreted as relations that are not necessarily functions.
\begin{definition} %% Let $\sigma$ and $\pi$ be $L^{min}$ formulas (used to denote substitutes for addition and multiplication respectively).
  Let $\sigma$ and $\pi$ be approximations to the sum and product. Then the {\em coarse interpretation based on $\sigma$ and $\pi$}, denoted $\mathfrak{C}_{\sigma,\pi}$ or just $\mathfrak{C}$ when $\sigma,\pi$ are clear from the context, maps each formula $\phi \in L_p^{min}$ to $[ \phi ]$ in $L_p/Z$, furthermore for the rest of $L_p$:
  \begin{description}
  %% \item $(s(x) = y)^{\mathfrak{C}} = \left[ \displaystyle\bigvee_{0 \leq a < p} (R_a(x) \wedge R_{(a+1 \mod p)}(y) ) \right] \wedge \left[ x \leq y \right]$
  %% \item $(s(x) < y)^{\mathfrak{C}} = \exists z ((s(x) = z)^{\mathfrak{C}} \wedge [z < y]^{\mathfrak{C}})$
  %% \item $(s(x) > y)^{\mathfrak{C}} = \exists z ((s(x) = z)^{\mathfrak{C}} \wedge [z > y]^{\mathfrak{C}})$
  \item $(x + y = z)^{\mathfrak{C}} = \left[ \displaystyle\bigvee_{0 \leq a < p} \left( \bigvee_{0 \leq b < p} R_a(x) \wedge R_b(y) \wedge R_{a + b \mod p}(z) \right) \right] \wedge \left[ \sigma(x,y,z) \right]$
  \item $(x + y < z)^{\mathfrak{C}} =  \exists u ((x + y = u)^{\mathfrak{C}} \wedge [u < z])$
%%  \item $(x + y > z)^{\mathfrak{C}} =  \exists u ((x + y = u)^{\mathfrak{C}} \wedge [u > z])$
  \item $(x \cdot y = z)^{\mathfrak{C}} = \left[ \displaystyle\bigvee_{0 \leq a < p} \left( \bigvee_{0 \leq b < p} (R_a(x) \wedge R_b(y) \wedge R_{a \cdot b \mod p}(z)) \right) \right] \wedge [\pi(x,y,z)]$
  \item $(x \cdot y < z)^{\mathfrak{C}} =  \exists u ((x \cdot y = u)^{\mathfrak{C}} \wedge [u < z])$
%%  \item $(x \cdot y > z)^{\mathfrak{C}} =  \exists u ((x \cdot y = u)^{\mathfrak{C}} \wedge [u > z])$
  \item The first-order connectives $\vee$, $\wedge$, $\neg$, $\exists$ and $\forall$ are defined fairly straight forwardly, so for example $(\phi \vee \psi)^{\mathfrak{C}} = \phi^{\mathfrak{C}} \vee \psi^{\mathfrak{C}}$ and $(\exists x \phi)^{\mathfrak{C}} = \exists x (\phi)^{\mathfrak{C}}$ in the latter case $(\phi)^{\mathfrak{C}} = [\psi]$ for some $\psi \in L_p^{min}$ whilst $\phi \in L_p$. 

  %% \item Addition is defined using a family formulas  $\{ P_a(x,y,z) \}_{-p < a < p}$
  %%   \begin{itemize}
  %%   \item $P_0 (x,y,z)$ denotes the formula $y = 0 \wedge x = z$
  %%   \item If $a \geq 0$ then $P_{a + 1}(x,y,z)$ denotes the formula $y > 0 \wedge \exists v w (s(v) = y \wedge s(w) = z \wedge P_a(x,v,w))$
  %%   \item If $a \leq 0$ then $P_{a - 1}(x,y,z)$ denotes the formula $y < 0 \wedge \exists v w (s(y) = v \wedge s(z) = w \wedge P_a(x,v,w))$
  %%   \item Now $(x + y = z)$ denotes the formula $\displaystyle\bigvee_{-p < a < p} \left[ P_a (x,y,z) \right]$
  %%   \end{itemize}
  %%   Likewise multiplication is defined using a family of formulas $\{ T_a \}_{-p < a < p}$
  %%   \begin{itemize}
  %%   \item $T_0(x,y,z)$ denotes the formula $y = 0 \wedge z = 0$
  %%   \item If $a \geq 0$ then $T_{a + 1}(x,y,z)$ denotes the formula $(y > 0 \wedge \exists v w (s(v) = y \wedge w + x = z \wedge T_a(x,v,w)$).
  %%     This is the flat version of $y > 0 \wedge \exists v (x \cdot v + x = z \wedge x \cdot s(v) = z)$
  %%   %% TODO \item If $a \leq 0$ then $T_{a - 1}(x,y,z)$ denotes the formula $(y < 0 \wedge \exists v w (s(y) = v \wedge (x + z = w)^C \wedge T_a(x,v,w))$. This is the flat version of $y < 0 \wedge x + $
  %%   \end{itemize}
  \end{description}
  
\end{definition}

\section{The main results}
The first result is on computationally checking whether given positive formulas are satisfiable or not in the proposed generalisations of the integers modulo $p$.
\begin{theorem}   Let $\sigma$ and $\pi$ be approximations to the sum and product. Then the coarse interpretation based on $\sigma$ and $\pi$ is decidable.
  %% That is to say that there is a decicion procedure for deciding wheter the coarse interpretation sends a given formula to the bottom of $L_p/Z$.
\end{theorem}
{\bf Proof:} This follows since the coarse interpretation maps $L_p$ formulas to $L_p^{min}$ formulas and since $L_p^{min}$ is decidable.
\hfill Q.E.D.

The second result is a formal way of stating that the coarse interpretations are approximations to the standard interpretation.
\begin{theorem} Let $\sigma$ and $\pi$ be approximations to the sum and product. Then the precise interpretation of a positive $L_p$ formula, is equal to or below the coarse interpretation (based on $\sigma$ and $\pi$) in the Lindenbaum algebra, $L_p/Z$.
\end{theorem}
\textbf{Proof:} Consider the Lindenbaum algebra, $L_p/Z$. For formulas in $L_p^{min}$ this is clear as the interpretations are the same. For atomic formulas in $L_p$:
\begin{description}
%% \item $(s(x) = y)^{\mathfrak{C}} = \left[ \displaystyle\bigvee_{0 \leq a < p-1} (R_a(x) \wedge R_{(a+1 \mod p)}(y)) \wedge x \leq y \right]$ \\
%%   $>$ $\left[ \displaystyle\bigvee_{0 \leq a < p-2} (R_a(x) \wedge R_{(a+1 \mod p)}(y)) \wedge x = y \right] \vee \left[ R_{p-1}(x) \wedge R_0(y) \wedge s(x) = y \right]$
%%   $=$ $\left[ s(x) = y \right]$ the latter being the standard interpretation.
%% \item $(s(x) < y)^{\mathfrak{C}} = \exists z ((s(x) = z)^{\mathfrak{C}} \wedge [z < y]^{\mathfrak{C}}) > \left[ x < y \right]$ the latter follows since we just proved that $(s(x) = z)^{\mathfrak{C}}$ is below $[ s(x) = z ]$ and since $[ x < y]^{\mathfrak{C}} = [ x < y ]$  
%%   \item $(s(x) > y)^{\mathfrak{C}} = \exists z ((s(x) = z)^{\mathfrak{C}} \wedge [z > y]^{\mathfrak{C}}) > \left[ x < y \right]$ for the same reason as above.
\item $(x + y = z)^{\mathfrak{C}} = \left[ \displaystyle\bigvee_{0 \leq a < p} \left( \bigvee_{0 \leq b < p} R_a(x) \wedge R_b(y) \wedge R_{a + b \mod p}(z) \right) \right] \wedge [\psi(x,y,z)]$. The precise interpretation is clearly below each of the conjuncts, so by monotonicity
  of conjunction the precise interpretation is below. 
    %% where $\psi(x,y,z)$ denotes the formula
    %% $((0 < y \wedge x \leq z) \vee (y = 0 \wedge x = z) \vee (y < 0 \wedge z \leq x))$.
    %% The formula $\psi$ is designed so as to contain the graph of the addition function, see figure \ref{fig:addition}
  \item $(x + y < z)^{\mathfrak{C}} =  \exists u ([\sigma(x,y,u)] \wedge [u < z])$. This follows from the monotonicity of $\exists$,$\wedge$ the assumption on $\sigma$ and the fact that $x > z$ is in $L_p^{min}$ so its interpretation is precise.
  \item $(x + y > z)^{\mathfrak{C}} =  \exists u ((x + y = u)^{\mathfrak{C}} \wedge [u > z])$. This case is as above.
  \item $(x \cdot y = z)^{\mathfrak{C}} = \left[ \displaystyle\bigvee_{0 \leq a < p} \left( \bigvee_{0 \leq b < p} (R_a(x) \wedge R_b(y) \wedge R_{a \cdot b \mod p}(z)) \right) \right] \wedge [\psi(x,y,z)]$. This case is similar to addition.
    %% where $\psi(x,z)$ denotes the formula
    %% $((0 < y \wedge x \leq z) \vee (0 = y \wedge z = 0) \vee (0 = x \wedge z = 0) \vee (y < 0 \wedge z \leq y) \vee (x < 0 \wedge y < 0 \wedge z > 0))$
    %% The formula $\psi$ is designed so as to contain the graph of the multiplication function, see figure \ref{fig:multiplication}
  \item $(x \cdot y < z)^{\mathfrak{C}} =  \exists u ((x \cdot y = u)^{\mathfrak{C}} \wedge [u < z])$. Proven as the corresponding case for addition.
  \item $(x \cdot y > z)^{\mathfrak{C}} =  \exists u ((x \cdot y = u)^{\mathfrak{C}} \wedge [u > z])$. Proven as above.
  \item For $\vee$ and $\wedge$ this follows from monotonicity.
  \item For $\forall$ and $\exists$ this also follows from monotonicity.
\end{description}
\hfill{Q.E.D.}

\begin{corrollary}
If a diophantine inequality is unsolvable over some coarse interpretation then it is unsolvable in the standard sense.
\end{corrollary}
\textbf{Proof:} As mentioned earlier, if an equivalence class is below an other in the Lindenbaum algebra then its interpretation in the Tarskian sense is included in the other. Now if a given diophantine equation, i.e. a positive first-order fomlula, is unsolvable over some coarse interpretation then it is mapped to the bottom of $L^{min}_p/Z$, and this is the equivalence-class that corresponds to the empty set (of assignments of integers to variables). Since, by the last theorem, the standard interpretation is below the coarse interpretation, the interpretation of the given formula is the empty-set, i.e. it is unsolvable.
\hfill{Q.E.D.}

\section{Final remarks and extensions}
We have seen that we can decide whether diophantine inequalities are solvable in coarse interpretations. I believe that decision procedures can be implemented on a personal computer and used to study diophantine inequalities much as one studies diophantine equations by interpreting them in the integers modulo $p$, see e.g. the introductory chapters of the classic book by L. J. Mordell \cite{mordell1969diophantine}. Search for interpretations in algebras isomorphic to sub algebras of Lindenbaum algebras has been implemented and shown to work for formulas up to at least 3 variables, Rognes\cite{rognes2009turning} 

It should be possible to expand the $L_p^{min}$ with symbols for linear functions, such as $2x, 3x, 2x+1, 3x+1$ etc, keeping coarse interpretations decidable using e.g. a decision procedure for Presburger arithmetic. It remains open whether the latter extension would result in a more powerful way of occasionally establishing insolvability. Also computation in the Lindenbaum algebra may be more time consuming due to higher complexity.

The central properties of the present constructions are that we have a small language whose standard interpretation is decidable and a finite partition of the standard model, here the residue classes. It should be possible to carry the constructions over to expansions of other decidable theories with finitely partitioned standard models.

%% FØLGENDE MÅ PRESISERES
%% If the set of $L_p^{min}$ sentences true of the integers allows quantifier elimination then coverings defined by iteration (p.r.) are definable by a finite
%% number of iterates.

%% FØLGENDE ER AKTUELT HVIS EN ELLER ANNEN FORM FOR ITERASJON (P.R.?) ER MULIG.
%% In regards to the power, I propose a conjecture analogous to an open problem, known as the Skolem problem, regarding linear recurrence relations, see e.g. \cite{akshay2024robustness} for a fairly updated  overview into the state of problems related to the Skolem problem, including problems involving inequalities.
%% \begin{conjecture}
%% The coarse interpretations based on $\sigma$ and $\pi$, as defined in example \ref{minimal}, suffice for establishing the unsolvability of any unsolvable system of linear equations and inequaities over the integers.
%% \end{conjecture}

\bibliographystyle{plain}
\bibliography{gen_int_p}

\end{document}